\documentclass[11pt]{article}
\usepackage{lmodern}
\usepackage{lmodern}
\usepackage[T1]{fontenc}
\usepackage[utf8]{inputenc}
\usepackage{color}
\usepackage{mathtools}
\usepackage{enumitem}
\usepackage{amsmath}
\usepackage{amsthm}
\usepackage{amssymb}
\usepackage{geometry}
\usepackage{setspace}
\usepackage[authoryear,round]{natbib}
\usepackage{microtype}
\usepackage[bookmarks=false,
 breaklinks=false,pdfborder={0 0 1},backref=section,colorlinks=true]
 {hyperref}
\hypersetup{
 linkcolor=blue!50!black,citecolor=blue!50!black,urlcolor=blue!50!black}

\makeatletter

\@ifundefined{date}{}{\date{}}
\usepackage{amsthm}\usepackage{mathrsfs}
\usepackage{bm}
\usepackage{booktabs}
\usepackage{tabularx}\usepackage{array}\usepackage{longtable}\usepackage{xcolor}
\usepackage{float}
\usepackage{url}

\newtheorem{theorem}{Theorem}[section]
\newtheorem{proposition}[theorem]{Proposition}\newtheorem{corollary}[theorem]{Corollary}\newtheorem{assumption}{Assumption}[section]\theoremstyle{definition}
\newtheorem{definition}[theorem]{Definition}\theoremstyle{remark}
\newtheorem{remark}[theorem]{Remark}

\title{On rates of convergence for sample average approximations without smoothness}
\author{Hien Duy Nguyen$^{1,2}$, Jacob Westerhout$^3$, and Xin Guo$^3$\\[0.4em]
{\small $^1$School of Computing, Engineering, and Mathematical Sciences,}\\
{\small La Trobe University, Bundoora, VIC 3086, Australia;}\\
{\small $^2$Institute of Mathematics for Industry, Kyushu University,}\\
{\small Nishi Ward, Fukuoka 819-0395, Japan}\\[0.2em]
{\small $^3$School of Mathematics and Physics,}\\
{\small The University of Queensland, St Lucia, QLD 4072, Australia}}

\makeatother

\begin{document}
\maketitle 

\begin{abstract}
Sample average approximation (SAA) replaces an intractable expected
objective by an empirical average and is a basic device of modern
stochastic optimization. We develop a rate theory for optimal values
and empirical $\varepsilon$-minimizers that does not assume continuity,
lower semicontinuity, or smooth perturbation structure of the sample
objectives. Working on $\ell^{\infty}(X)$ with the Hoffmann--Jørgensen
outer-probability formalism, we show that uniform control of the empirical
objective process transfers deterministically to convergence rates
for optimal values, excess risks of empirical $\varepsilon$-minimizers,
and, under a sharp-growth condition, distances to the expected objective
solution set. Combined with the directional differentiability of the
infimum functional, this yields weak limits for empirical optimal
values at the $n^{-1/2}$ scale. Combined with LILs and maximal inequalities,
it yields outer almost-sure and outer-mean rates. The definability, envelope,
and VC-subgraph hypotheses are verified for definable discontinuous
or non-Lipschitz classes arising in direct $0$--$1$ classification,
fixed-architecture neural networks, threshold regression, and non-Lipschitz
$\ell_{p}$-type objectives with rational $0<p<1$. Practical sufficient conditions for 
measurability hypotheses are discussed.
Together, the framework extends continuity-based SAA
theory to a tame-topological setting. 
\end{abstract}


\section{Introduction}

Sample average approximation (SAA) is a central tool in stochastic
optimization. One starts from a problem of the form 
\begin{equation}
\inf_{x\in X}f(x),\qquad f(x)=\mathbb{E}[h(x,\xi)],\label{eq:intro-sp}
\end{equation}
where $X\subseteq\mathbb{R}^{d}$ is a feasible set, $\xi$ is a random
element, and the expectation in \eqref{eq:intro-sp} is unavailable
in closed form or too expensive to evaluate directly. Given independent
copies $\xi_{1},\dots,\xi_{n}$ of $\xi$, one replaces $f$ by the
empirical objective 
\begin{equation}
\hat{f}_{n}(x)=\frac{1}{n}\sum_{i=1}^{n}h(x,\xi_{i})\label{eq:intro-saa}
\end{equation}
and studies the approximate optimization problem $\inf_{x\in X}\hat{f}_{n}(x)$.
This is the standard SAA paradigm treated systematically in Chapter~5
of \citet{ShapiroDentchevaRuszczynski2021} and surveyed in \citet{KimPasupathyHenderson2015}.
The method is ubiquitous within stochastic programming, simulation
optimization, statistical learning, empirical-risk minimization, and
data-driven decision making.

The classical asymptotic theory is now well established in smooth
settings. In the framework of perturbation analysis for stochastic
programs, \citet{Shapiro1991} obtained weak convergence for optimal
values at the $n^{-1/2}$ scale. Under additional differentiability,
uniqueness, and second-order regularity assumptions, one obtains the
same order for minimizers and solution mappings; see again \citet[Chapter~5]{ShapiroDentchevaRuszczynski2021}.
A shortcoming of these results is that they are designed for objectives
that admit smooth local expansions and therefore do not cover discontinuous
objective classes or other pathologies such as unbounded derivatives.

A substantial refinement was obtained by \citet{Banholzer2022}.
Working in the Banach spaces $C(X)$ and $C^{1}(X)$, and using Banach-space
compact and bounded LILs due to Kuelbs \citep[Theorem~4.4]{Kuelbs1976}
and \citet[Theorem~4.1]{Kuelbs1977}, they proved almost-sure and
mean convergence rates for the empirical objective process, for optimal
values, and, under growth conditions, for minimizers. Their analysis
is an important advance that moves beyond weak convergence and yields
explicit almost-sure and mean statements. At the same time, the basic
hypotheses remain continuity-based: the sample objectives are embedded
into spaces of continuous or continuously differentiable functions,
and assumptions such as Lipschitz continuity of $h(\cdot,\xi)$ are
structural in the framework.

The present manuscript starts from the observation that, at the level
of optimal values, continuity is not necessary. The quantity of interest
is the infimum functional on the space of bounded objectives. In the
recent work of \citet{Westerhout2024}, the minimum empirical risk
is analyzed directly as the image of the empirical process under the
infimum map on $\ell^{\infty}(X)$, the Banach space of bounded real-valued
functions on $X$ equipped with the supremum norm. The key observation
is that the infimum map is Hadamard directionally differentiable on
$\ell^{\infty}(X)$ (cf. \citealp[Theorem~3.1]{Westerhout2024}).
This viewpoint naturally accommodates the Hoffmann--Jørgensen outer-probability
formalism (see e.g. \citealp{vdVW2023}), which is the standard technical
setting for empirical processes indexed by nonseparable or discontinuous
classes without requiring Borel measurability as $\ell^{\infty}(X)$-valued
random variables.

Our aim is to combine these two strands of research. We keep the infimum-centered
$\ell^{\infty}(X)$ viewpoint of \citet{Westerhout2024} for the value
process, and we combine it with deterministic perturbation bounds
in the spirit of the optimization-transfer arguments underlying \citet{Banholzer2022}.
The outcome is a theory in which any uniform bound on the empirical
objective process---weak, in outer probability, outer almost surely, or in
outer mean---immediately yields the corresponding rate for the optimal
value and for empirical $\varepsilon$-minimizers. Exact minimizer
localization is then recovered as an optional refinement under a sharp-growth
condition. In particular, the core results do not assume continuity,
lower semicontinuity, or even attainment of the minimum.

Our second tool is o-minimal definability. The optimization and variational
analysis literature have long recognized definability in o-minimal
structures on the ordered field of real numbers as a natural regularity
framework; see \citet{Ioffe2009}, \citet{BolteDaniilidisLewis2009},
and \citet{BoltePauwels2016}. One reason is that definable families
are tame and have finite combinatorial complexity. In particular,
Chapter~5 of \citet{vdDries1998} shows that definable set systems
have finite Vapnik--Chervonenkis (VC) dimension; see also \citet[Section~2.6]{vdVW2023}
for the empirical-process consequences of VC-subgraph complexity.
Together with the standard structural results summarized in \citet{LeLoi2010}
and \citet[Chapter~8]{Pila2022}, this provides a route to empirical-process
control that does not depend on continuity of the sample paths. The
route is especially effective for discontinuous classes built from
inequalities, equalities, or threshold rules.

This perspective is reinforced by explicit quantitative bounds such
as \citet[Theorems~8.4 and~8.14]{AnthonyBartlett2009} which prove
polynomial VC-dimension bounds for classes within o-minimal structures
that are computed by finite arithmetic, comparison, indicator, and
exponential programs. Further, \citet{MontanaPardo2009} obtain analogous
polynomial bounds for concept classes generated by Pfaffian trees
and circuits. Consequently, direct $0$--$1$ classification losses
for perceptrons \citep{LiLin2007,NguyenSanner2013}, fixed-architecture
neural networks \citep[Section~7.2]{Westerhout2024}, threshold-regression
criteria \citep{Hansen2000,KoulQianSurgailis2003}, and robust sparse
nonconvex regression losses \citep{YeYingShaoLiChen2017,BucciniEtAl2020}
can be placed within an o-minimal, finite-complexity setting, despite
being discontinuous or nonsmooth in the optimization parameter.

The main contributions of the manuscript are therefore the following.
At the abstract level, the paper formulates SAA entirely on $\ell^{\infty}(X)$
and proves deterministic perturbation inequalities showing that sup-norm
fluctuations of the empirical objective control optimal values, $\varepsilon$-minimizers,
and, under sharp growth, exact solution-set distances. At the probabilistic
level, the generalized delta method for directionally differentiable
maps transfers weak convergence of the empirical process to weak convergence
of the empirical optimal value, thereby recovering $n^{-1/2}$ outer-probability
rates without smoothness. The same deterministic perturbation mechanism
converts LILs and maximal inequalities into outer almost-sure and outer-mean
rate statements. At the model-verification level, o-minimal definability
and finite VC-subgraph complexity make it possible to check the definability,
envelope, and VC-subgraph hypotheses for discontinuous optimization
classes, including square-integrably enveloped classes for which the
weak, outer-mean, and LIL inputs all remain available. Relevant
measurability hypotheses are also discussed, along with their verification.

The manuscript proceeds as follows. Section~\ref{sec:prelim} collects
the mathematical preliminaries on outer probability, the infimum functional,
VC-subgraph classes, o-minimal structures, restricted analytic and
Pfaffian functions, and the empirical-process results that will be
used later. Section~\ref{sec:setup} formulates the SAA problem and
fixes the standing notation. Section~\ref{sec:main} contains the
deterministic perturbation theorems, the weak-limit theorem for optimal
values, and the rate consequences. Section~\ref{sec:applications}
develops explicit discontinuous and non-Lipschitz examples. Section~\ref{sec:conclusion}
closes with a brief discussion of scope and possible extensions. Proofs and technical results
are collected in Appendix~\ref{app:proofs}.

\section{Mathematical preliminaries}

\label{sec:prelim}

\subsection{\texorpdfstring{Outer probability, outer expectation, and $\ell^{\infty}(X)$}{Outer
probability, outer expectation, and l-infinity(X)}}

Let $(\Omega,\mathcal{F},\mathbb{P})$ be a probability space. For
a possibly nonmeasurable map $U:\Omega\to\overline{\mathbb{R}}=\mathbb{R}\cup\{\pm\infty\}$,
the outer expectation is defined by 
\begin{equation}
\mathbb{E}^{*}U=\inf\bigl\{\mathbb{E}[V]:V\ge U,\;V\text{ measurable, and }\mathbb{E}[V]\text{ exists}\bigr\},\label{eq:outer-expectation}
\end{equation}
and the outer probability of a set $A\subseteq\Omega$ is 
\[
\mathbb{P}^{*}(A)=\mathbb{E}^{*}\mathbf{1}_{A}.
\]
This is the Hoffmann--Jørgensen formalism used by \citet{vdVW2023}.
If $(U_{n})$ is a sequence of real maps on $\Omega$ and $(r_{n})$
is a positive deterministic sequence, we write 
\[
U_{n}=O_{\mathbb{P}^{*}}(r_{n})
\]
if for every $\varepsilon>0$ there exists $M<\infty$ such that 
\[
\limsup_{n\to\infty}\mathbb{P}^{*}\bigl(\lvert U_{n}\rvert>Mr_{n}\bigr)\le\varepsilon,
\]
and similarly $U_{n}=o_{\mathbb{P}^{*}}(r_{n})$ if $U_{n}/r_{n}\to0$
in outer probability.

Following \citet[Lemma~1.2.1]{vdVW2023}, write $U^{*}$ for
a version of the minimal measurable majorant of a real map $U$. For
outer almost-sure order statements involving possibly nonmeasurable
maps, following \citet[Definition~1.9.1 and Lemma~1.9.2]{vdVW2023}, we write
\[
U_{n}=O_{\mathrm{as}^{*}}(r_{n})
\]
if, for some versions of the minimal measurable majorants,
\[
\limsup_{n\to\infty}\frac{\lvert U_{n}\rvert^{*}}{r_{n}}<\infty\qquad\text{almost surely}.
\]
The notation $U_{n}=o_{\mathrm{as}^{*}}(r_{n})$ is defined analogously
with the displayed limsup replaced by convergence to zero. When the
$U_{n}$ are measurable, these are the ordinary almost-sure order
relations. For sequences this is the rate analogue of the outer almost-sure
mode that is equivalent to almost uniform convergence in \citet[Lemma~1.9.2]{vdVW2023}.
For iterated logarithms we write 
\[
\mathrm{LL}(u)=1\vee\log\log(u\vee e^{e}),\qquad u\ge0.
\]

For any nonempty set $A$, the symbol $\ell^{\infty}(A)$ denotes
the Banach space of all bounded real-valued functions $g:A\to\mathbb{R}$,
equipped with the norm 
\[
\lVert g\rVert_{\infty}=~\sup_{a\in A}\lvert g(a)\rvert.
\]
Thus $\ell^{\infty}(A)$ is simply the space of bounded functions
on $A$ with the supremum norm. Weak convergence in $\ell^{\infty}(A)$
is understood in the Hoffmann--Jørgensen sense of
\citet[Definition~1.3.3]{vdVW2023}; the bounded-function-space
setting is discussed in \citet[Section~1.5]{vdVW2023}. In this mode of
weak convergence, the empirical process is not assumed to be Borel
measurable as an $\ell^{\infty}(A)$-valued random element.

\subsection{\texorpdfstring{The infimum functional and $\varepsilon$-minimizers}{The
infimum functional and epsilon-minimizers}}

Let $X$ be a nonempty set. For $g\in\ell^{\infty}(X)$ define the
infimum functional 
\begin{equation}
\iota(g)=\inf_{x\in X}g(x).\label{eq:iota-def}
\end{equation}
Given $f\in\ell^{\infty}(X)$ and $\varepsilon\ge0$, define the $\varepsilon$-minimizer
set 
\begin{equation}
S_{f}^{\varepsilon}=\{x\in X:f(x)\le\iota(f)+\varepsilon\}.\label{eq:Sfeps}
\end{equation}
The set $S_{f}^{\varepsilon}$ is nonempty for every $\varepsilon>0$,
whereas $S_{f}^{0}$ may be empty when the infimum is not attained.

The basic structural result is the directional differentiability of
the infimum map; i.e. Theorem~3.1
by \citet{Westerhout2024}.

\begin{theorem}[Directional differentiability of the infimum map]\label{thm:infdir}
Let $X$ be a nonempty set, let $\iota:\ell^{\infty}(X)\to\mathbb{R}$
be defined by \eqref{eq:iota-def}, and let $f,g\in\ell^{\infty}(X)$.
Then $\iota$ is $1$-Lipschitz and Hadamard directionally differentiable
at $f$. Its directional derivative in direction $g$ is 
\begin{equation}
\iota_{f}'(g)=\lim_{\varepsilon\downarrow0}\inf_{x\in S_{f}^{\varepsilon}}g(x).\label{eq:dir-derivative}
\end{equation}
\end{theorem}

\begin{remark} Expression \eqref{eq:dir-derivative} is a
general formula without any topological assumptions on $X$. If $X$
is compact, $f$ is lower semicontinuous, and the limiting process
has lower semicontinuous sample paths, then \citet[Remark~4.3]{Westerhout2024}
show that the weak-limit expression simplifies to an infimum over
exact minimizers. \end{remark}

We shall also use the corresponding directionally differentiable delta
method in the Hoffmann--Jørgensen setting. The following is the form
of \citet[Fact~3.2]{Westerhout2024} needed below.

\begin{theorem}[Directional delta method]\label{thm:directional-delta}
Let $U,V$ be normed vector spaces, let $\mu\in U$, and let
$\phi:U\to V$ be Hadamard directionally differentiable at $\mu$.
For each $n\in\mathbb{N}$, let $Y_{n}:\Omega\to U$ be maps, and
let $Y:\Omega\to U$ be Borel measurable. If $\tau_{n}\to\infty$ and
\[
\tau_{n}(Y_{n}-\mu)\rightsquigarrow Y,
\]
then
\[
\tau_{n}(\phi(Y_{n})-\phi(\mu))\rightsquigarrow\phi_{\mu}'(Y)
\]
and
\[
\tau_{n}(\phi(Y_{n})-\phi(\mu))-\phi_{\mu}'(\tau_{n}(Y_{n}-\mu))=o_{\mathbb{P}^{*}}(1),
\]
where the latter display should be read as convergence to zero in outer probability
with respect to the norm on $V$.
\end{theorem}

When one wants to convert value perturbations into exact minimizer
localization, one needs a deterministic growth condition.

\begin{definition}[Sharp growth]\label{def:sharp} Assume now that
$X\subseteq\mathbb{R}^{d}$ and let $f\in\ell^{\infty}(X)$ be such
that $S_{f}^{0}\neq\varnothing$. For $\kappa\ge1$ we say that $f$
satisfies a sharp-growth inequality of order $\kappa$ if there exists
$\alpha>0$ such that 
\begin{equation}
f(x)-\iota(f)\ge\alpha\,\operatorname{dist}(x,S_{f}^{0})^{\kappa}\qquad\text{for all }x\in X,\label{eq:sharp-growth}
\end{equation}
where the distance is Euclidean. \end{definition}

\begin{remark} Inequality \eqref{eq:sharp-growth} is an error-bound-type
condition. When $\kappa=2$ it is a global version of the quadratic-growth
or strong-local-minimum conditions; see, for example, the discussion
of strong local minimizers and error bounds in \citet[Section~8.1]{CuiPang2022}.
\end{remark}

\subsection{VC-subgraph classes and empirical-process inputs}
\label{subsec:vc}

Let $\mathcal{H}$ be a class of real-valued functions on a measurable
space $(\Xi,\mathcal{G})$. The class is called a VC-subgraph class
if the collection of strict subgraphs 
\[
\bigl\{(\xi,t)\in\Xi\times\mathbb{R}:h(\xi)>t\bigr\},\qquad h\in\mathcal{H},
\]
has finite VC dimension. For $\{0,1\}$-valued classes this reduces
to the ordinary VC property of the associated concept class. This is
a set-theoretic complexity condition; the measurability hypotheses
needed for empirical-process arguments are stated separately below.

Throughout this subsection, $\xi_{1},\dots,\xi_{n}$ denotes an i.i.d.
$P$-sample and 
\[
\mathbb{P}_{n}=\frac{1}{n}\sum_{i=1}^{n}\delta_{\xi_{i}}
\]
denotes the empirical measure. Following \citet[Definition~2.3.3]{vdVW2023},
$\mathcal{H}$ is a $P$-measurable class if for every $n\in\mathbb{N}$
and every $e=(e_{1},\dots,e_{n})\in\mathbb{R}^{n}$ the map 
\[
(\xi_{1},\dots,\xi_{n})\longmapsto\sup_{h\in\mathcal{H}}\left|\sum_{i=1}^{n}e_{i}h(\xi_{i})\right|
\]
is measurable for the $P^{n}$-completion of $\mathcal{G}^{n}$. A class
$\mathcal{H}$ is pointwise measurable if there exists a countable subclass
$\mathcal{H}_{0}\subseteq\mathcal{H}$ such that for every $h\in\mathcal{H}$
there is a sequence $(h_{m})$ in $\mathcal{H}_{0}$ with $h_{m}(\xi)\to h(\xi)$
for every $\xi\in\Xi$; see \citet[Example~2.3.4]{vdVW2023}.
For parameter-indexed classes $\mathcal{H}=\{h_{x}:x\in X\}$, a convenient
sufficient condition is the existence of a countable subset $X_{0}\subseteq X$
such that every $x\in X$ admits a sequence $(x_{m})$ in $X_{0}$
with $h_{x_{m}}(\xi)\to h_{x}(\xi)$ for every $\xi\in\Xi$. We also write
\[
\mathcal{H}_{\infty}=\{h_{1}-h_{2}:h_{1},h_{2}\in\mathcal{H}\},\quad
\mathcal{H}_{\delta}=\{u\in\mathcal{H}_{\infty}:\lVert u\rVert_{P,2}<\delta\},\quad
\mathcal{H}_{\infty}^{2}=\{u^{2}:u\in\mathcal{H}_{\infty}\},
\]
where $\lVert u\rVert_{P,2}=(Pu^{2})^{1/2}$, with infinite values allowed.

\begin{assumption}[Empirical-process measurability]\label{ass:ep-meas}
We say that $\mathcal{H}$ satisfies the empirical-process measurability
hypotheses if the following hold:
\begin{enumerate}[label=(\roman*),itemsep=0.25em]
\item $\mathcal{H}$ is a $P$-measurable class.
\item $\mathcal{H}_{\delta}$ and $\mathcal{H}_{\infty}^{2}$ are $P$-measurable classes
for every $\delta>0$.
\item $\sup_{h\in\mathcal{H}}\lvert(\mathbb{P}_{n}-P)h\rvert$ is measurable
for every $n\in\mathbb{N}$.
\end{enumerate}
\end{assumption}

The empirical-process facts needed later are standard consequences
of entropy bounds for VC-subgraph classes. Recall that a class $\mathcal{H}$
is called $P$-Donsker if the empirical process 
\[
\mathbb{G}_{n}h=\sqrt{n}(\mathbb{P}_{n}-P)h,\qquad h\in\mathcal{H},
\]
viewed as a random element of $\ell^{\infty}(\mathcal{H})$, converges
weakly in the Hoffmann--Jørgensen sense to a tight centered Gaussian
process $\mathbb{G}_{P}$ on $\mathcal{H}$ with covariance 
\[
\operatorname{Cov}\bigl(\mathbb{G}_{P}h_{1},\mathbb{G}_{P}h_{2}\bigr)=P(h_{1}h_{2})-Ph_{1}\,Ph_{2}.
\]

\begin{theorem}[Donsker theorem for VC-subgraph classes]\label{thm:vc-donsker}
Let $\mathcal{H}$ be a VC-subgraph class of $\mathcal{G}$-measurable
functions. Suppose that $\mathcal{H}$ admits
a measurable envelope $H$, such that $\lvert h\rvert\le H$ for every
$h\in\mathcal{H}$, and that $PH^{2}<\infty$ for the probability measure
$P$ under consideration. Assume also that Assumption~\ref{ass:ep-meas}(i) and (ii)
hold. Then $\mathcal{H}$ is $P$-Donsker.
\end{theorem}

\begin{remark}\label{rem:ptwise-meas} This is a standard consequence of the VC covering
theorem \citep[Theorem~2.6.7]{vdVW2023} together with the uniform-entropy
Donsker theorem and its VC-subgraph corollaries \citep[Theorems~2.5.2 and~2.6.8]{vdVW2023}.
Under the square-integrable envelope assumptions imposed below, a
convenient sufficient condition for Assumption~\ref{ass:ep-meas}
is pointwise measurability. Indeed, \citet[Example~2.3.4]{vdVW2023}
gives that $\mathcal{H}$ is a $P$-measurable class. If $\mathcal{H}_{0}$ is a
countable pointwise dense subclass, then differences of pairs from
$\mathcal{H}_{0}$ pointwise approximate $\mathcal{H}_{\infty}$, and
dominated convergence under the envelope $2H$ preserves the strict
$L_{2}(P)$ restriction defining $\mathcal{H}_{\delta}$ after passing
to a subsequence. The squared-difference class is handled similarly
under the envelope $4H^{2}$. Finally, $\sup_{h\in\mathcal{H}}\lvert(\mathbb{P}_{n}-P)h\rvert$
is the supremum over the corresponding countable subclass and hence
measurable.
\end{remark}

\begin{remark}\label{rem:pullback} In our intended applications
the class is indexed by parameters, $\mathcal{H}=\{h_{x}:x\in X\}$.
If $\Pi:X\to\mathcal{H}$ is defined by $\Pi(x)=h_{x}$, then the
pullback operator 
\[
T:\ell^{\infty}(\mathcal{H})\to\ell^{\infty}(X),\qquad(Tu)(x)=u\bigl(\Pi(x)\bigr),
\]
is linear and continuous, with $\lVert Tu\rVert_{\infty}=\lVert u\rVert_{\infty}$.
Suppose now that $\mathbb{G}_{n}\rightsquigarrow\mathbb{G}$ in $\ell^{\infty}(\mathcal{H})$.
Since $T$ is continuous, the continuous mapping theorem for Hoffmann--Jørgensen
weak convergence gives 
\[
T\mathbb{G}_{n}\rightsquigarrow T\mathbb{G}\qquad\text{in }\ell^{\infty}(X).
\]
This is exactly the desired pulled-back convergence, because for every
$x\in X$, 
\[
(T\mathbb{G}_{n})(x)=\mathbb{G}_{n}(\Pi(x))=\mathbb{G}_{n}(h_{x})=\sqrt{n}(\mathbb{P}_{n}-P)h_{x},
\]
and similarly $(T\mathbb{G})(x)=\mathbb{G}(h_{x})$. Thus Donsker
convergence on $\mathcal{H}$ transfers immediately to weak convergence
of the empirical objective process indexed by $X$. \end{remark}

For almost-sure rates we use the LIL of \citet{Yukich1990}
through a VC verification corollary. The underlying weak-entropy criterion
is recorded in Theorem~\ref{thm:yukich} in Appendix~\ref{app:proofs}.
Here we specialize to the VC-subgraph implication that will be used
repeatedly in the examples.

\begin{proposition}[Almost-sure order for VC-subgraph classes]\label{prop:vc-yukich}
Let $\mathcal{H}$ be a VC-subgraph class of $\mathcal{G}$-measurable
functions with measurable envelope $H\in L_{2}(P)$.
Assume also that Assumption~\ref{ass:ep-meas}(iii) holds. Then
\begin{equation}
\sup_{h\in\mathcal{H}}\lvert(\mathbb{P}_{n}-P)h\rvert=O\!\left(\sqrt{\frac{\log\log n}{n}}\right)\qquad\text{almost surely}.\label{eq:vc-blil-rate}
\end{equation}
\end{proposition}

\begin{remark} This result complements the Banach-space LILs used
by \citet{Banholzer2022} and extends their almost-sure theory to square-integrably
enveloped definable classes. \end{remark}

\subsection{O-minimal structures, restricted analytic functions, and Pfaffian
functions}

An o-minimal structure $\mathcal{R}$ on the ordered field $(\mathbb{R},+,\cdot,<)$
assigns to each $n\in\mathbb{N}$ a collection of subsets of $\mathbb{R}^{n}$,
called the definable sets of $\mathcal{R}$, such that the family
is closed under finite Boolean operations, Cartesian products, and
coordinate projections, contains algebraic sets, and has the one-dimensional
tameness property that every definable subset of $\mathbb{R}$ is
a finite union of points and open intervals. A function is definable
if its graph is definable \citep[Section~1]{LeLoi2010}.

The intuitive content of definability is that sets and functions are
built from a fixed dictionary of primitive relations and functions
using finitely many logical operations and projections. In one variable
this rules out arbitrarily complicated oscillation, accumulation of
isolated components, or fractal-type behavior; this is the tame topology.
In higher dimensions it yields a robust geometric calculus: images,
inverse images, closures, boundaries, and compositions of definable
objects remain definable \citep[Section~2]{LeLoi2010}.

A basic special case is the semialgebraic o-minimal structure, namely
the pure ordered-field structure $(\mathbb{R},+,\cdot,<)$. A subset
of $\mathbb{R}^{n}$ is called semialgebraic if it is obtained from
finitely many polynomial equalities and strict or weak polynomial
inequalities by finitely many Boolean operations. By the Tarski--Seidenberg
theorem, coordinate projections of semialgebraic sets are again semialgebraic;
hence semialgebraic sets are exactly the definable sets of this basic
structure \citep[Chapter~2]{vdDries1998}. Accordingly, semialgebraic
functions are precisely those whose graphs are semialgebraic. When
we later say that a set or function is definable in the semialgebraic
structure, we simply mean definable in this basic ordered-field structure.
This is the natural tame setting for models assembled from finitely
many arithmetic operations and polynomial inequality tests.

Two o-minimal structures are particularly useful. First, a
restricted analytic function on $\mathbb{R}^{m}$ is a function $g$
for which there exists a real-analytic function $\tilde{g}$ on a
neighborhood of $[-1,1]^{m}$ such that $g(x)=\tilde{g}(x)$ for $x\in[-1,1]^{m}$
and $g(x)=0$ outside $[-1,1]^{m}$. The o-minimal structure generated
by all restricted analytic functions is denoted $\mathbb{R}_{\mathrm{an}}$,
and adjoining the global exponential map yields the o-minimal structure
$\mathbb{R}_{\mathrm{an},\exp}$ \citep[Examples~1.5--1.7]{LeLoi2010}.

Second, let $U\subseteq\mathbb{R}^{m}$ be open. A finite family of
real-analytic functions $f_{1},\dots,f_{q}:U\to\mathbb{R}$ is called a Pfaffian
chain on $U$ if for each $i=1,\dots,q$ and each coordinate $j=1,\dots,m$
there exists a polynomial $P_{ij}$ such that 
\[
\frac{\partial f_{i}}{\partial x_{j}}(x)=P_{ij}\bigl(x,f_{1}(x),\dots,f_{i}(x)\bigr)\qquad\text{for all }x\in U.
\]
A Pfaffian function associated with this chain is any function of
the form 
\[
x\longmapsto Q\bigl(x,f_{1}(x),\dots,f_{q}(x)\bigr)
\]
with $Q$ polynomial. The Pfaffian structure described in
\citet[Example~1.8]{LeLoi2010} is generated by Pfaffian functions whose
chains are defined on all of $\mathbb{R}^{m}$, and is often denoted
$\mathbb{R}_{\mathrm{Pfaff}}$. More
general Pfaffian functions on open domains, as in
\citet[Section~3]{MontanaPardo2009}, are the objects used in the
Pfaffian complexity bounds below.

The relevance of these structures is that they remain stable under
the threshold and equality predicates that generate discontinuous
bounded losses. For instance, if $g(x,\xi)$ is definable, then so
are the indicator losses $\mathbf{1}_{\{g(x,\xi)>0\}}$, $\mathbf{1}_{\{g(x,\xi)\ge0\}}$,
and $\mathbf{1}_{\{g(x,\xi)=0\}}$, even though they are typically
discontinuous in $x$.

\begin{theorem}[Finite VC complexity of definable subgraph families]\label{thm:def-vc}
Let $\Theta\subseteq\mathbb{R}^{d}$ and $\Xi\subseteq\mathbb{R}^{m}$
be definable sets in an o-minimal structure on $(\mathbb{R},+,\cdot,<)$,
and let $h:\Theta\times\Xi\to\mathbb{R}$. Assume that the set 
\[
\{(\theta,\xi,t)\in\Theta\times\Xi\times\mathbb{R}:h(\theta,\xi)>t\}
\]
is definable. Equivalently, since definable sets are closed under
complement, assume that the total subgraph 
\[
\operatorname{Sub}(h)=\{(\theta,\xi,t)\in\Theta\times\Xi\times\mathbb{R}:h(\theta,\xi)\le t\}
\]
is definable. Then the family of strict subgraphs 
\[
\bigl\{(\xi,t)\in\Xi\times\mathbb{R}:h(\theta,\xi)>t\bigr\}_{\theta\in\Theta}
\]
has finite VC dimension. Equivalently, the class $\{h_{\theta}=h(\theta,\cdot):\theta\in\Theta\}$
is a VC-subgraph class. \end{theorem}

\begin{remark} This is an immediate application of the VC theory
of definable families developed in Chapter~5, Section~2 of \citet{vdDries1998}
to the definable family of strict subgraphs in the ambient space $\Xi\times\mathbb{R}$.
\end{remark}

\subsection{Quantitative VC dimension bounds}

For several model classes within appropriate o-minimal structures
one has more than qualitative finiteness of VC complexity. The following remarks record quantitative bounds that complement
the discussion above.

\begin{remark}[Arithmetic and exponential program bounds]\label{rem:AB-programs}
Let $g:\mathbb{R}^{m}\times\mathbb{R}^{n}\to\mathbb{R}$. Assume that
for each triple $(\theta,u,s)\in\mathbb{R}^{m}\times\mathbb{R}^{n}\times\mathbb{R}$
the predicate $g_{\theta}(u)>s$ can be decided, in the exact computational
model of \citet[Theorems~8.4 and~8.14]{AnthonyBartlett2009}, by a
decision procedure that takes $(\theta,u,s)$ as input and returns
$\mathbf{1}_{\{g_{\theta}(u)>s\}}$. 
\begin{enumerate}[label=(\roman*),itemsep=0.25em]
\item Suppose that $m\ge1$, $t\ge1$, and this decision procedure uses
at most $t$ steps consisting of the arithmetic operations $+,-,\times,/$
(with division only where defined), conditional jumps based on the
comparisons $>,\ge,<,\le,=,\neq$ of real numbers, and the output
instructions $0$ and $1$. Then the strict subgraph family 
\[
\bigl\{(u,s)\in\mathbb{R}^{n}\times\mathbb{R}:g_{\theta}(u)>s\bigr\}_{\theta\in\mathbb{R}^{m}}
\]
has VC dimension 
\begin{equation}
\operatorname{VCdim}\bigl(\{(u,s):g_{\theta}(u)>s\}_{\theta\in\mathbb{R}^{m}}\bigr)\le4m(t+2).\label{eq:AB-arith}
\end{equation}
\item Suppose that $m\ge1$ and the same computational model is augmented
by the exponential map. Then 
\begin{equation}
\operatorname{VCdim}\bigl(\{(u,s):g_{\theta}(u)>s\}_{\theta\in\mathbb{R}^{m}}\bigr)\le t^{2}m\bigl(m+19\log_{2}(9m)\bigr).\label{eq:AB-exp}
\end{equation}
If, more precisely, the exponential function is evaluated at most
$q$ times, then 
\begin{equation}
\operatorname{VCdim}\bigl(\{(u,s):g_{\theta}(u)>s\}_{\theta\in\mathbb{R}^{m}}\bigr)\le m^{2}(q+1)^{2}+11m(q+1)\bigl(t+\log_{2}(9m(q+1))\bigr).\label{eq:AB-exp-q}
\end{equation}
\end{enumerate}
\end{remark}

\begin{remark}[Pfaffian program bounds]\label{rem:MP}
Let $\{g_{\theta}:\theta\in\Theta\}$ be a family of real-valued
functions with parameter dimension $k$. Assume that for each triple
$(\theta,u,s)$ the predicate $g_{\theta}(u)>s$ can be decided, in
the exact Pfaffian sequential-program model of \citet[Theorem~4]{MontanaPardo2009},
by a program of running time $t\ge1$ using the arithmetic operations
$+,-,\times,/$ (with division only where defined), conditional jumps
based on equality and inequality tests of real numbers, at most $q$
Pfaffian evaluations, a fixed Pfaffian chain of length $\ell$, and
degree bound $D\ge2$. Let $q_{0}=\max\{q,\ell,1\}$. Then the
threshold family 
\[
\bigl\{(u,s):g_{\theta}(u)>s\bigr\}_{\theta\in\Theta}
\]
has VC dimension bounded, for a universal constant $C$, by 
\begin{equation}
\operatorname{VCdim}\bigl(\{(u,s):g_{\theta}(u)>s\}_{\theta\in\Theta}\bigr)\le C\left((q_{0}k)^{2}+k(q_{0}+t)\log_{2}((k+1)t)\log_{2}D\right).\label{eq:MP-bound}
\end{equation}
This is the sequential Pfaffian-program bound of \citet[Theorem~4]{MontanaPardo2009},
with $q_{0}$ used to dominate both the number of Pfaffian evaluations
and the chain length. \end{remark}

\section{Problem formulation}

\label{sec:setup}

Retain $(\Xi,\mathcal{G})$, the i.i.d. sample $\xi,\xi_{1},\xi_{2},\dots$,
its law $P$, and the empirical measure $\mathbb{P}_{n}$ from Subsection~\ref{subsec:vc}.
Let $X$ be a nonempty index set. In the abstract infimum theory of
\citet{Westerhout2024}, no topology on $X$ is needed. For solution-set
localization and for the definable examples, we will take
$X\subseteq\mathbb{R}^{d}$ and $\Xi\subseteq\mathbb{R}^{m}$; depending
on the application, $X$ may be the whole space $\mathbb{R}^{d}$,
a box such as $[-M,M]^{d}$, or any other definable subset of $\mathbb{R}^{d}$.

We consider a loss function 
\begin{equation}
h:X\times\Xi\to\mathbb{R},\label{eq:loss}
\end{equation}
such that each section $h_{x}(\cdot)=h(x,\cdot)$ is measurable and
integrable under $P$, and define the expected and empirical objectives
by 
\begin{equation}
f(x)=\mathbb{E}[h(x,\xi)]=Ph_{x},\qquad\hat{f}_{n}(x)=\frac{1}{n}\sum_{i=1}^{n}h(x,\xi_{i})=\mathbb{P}_{n}h_{x}.\label{eq:obj-defs}
\end{equation}
Throughout the main results we work in the bounded-objective regime
in which $f\in\ell^{\infty}(X)$ and the centered empirical objective
$\hat{f}_{n}-f$ belongs to $\ell^{\infty}(X)$ almost surely for
the relevant sample sizes. This is automatic for bounded loss classes
and, more generally, whenever the class $\mathcal{H}$ admits a measurable
envelope $H$ with $PH<\infty$, since then $|f(x)|\le PH$ and $|\hat{f}_{n}(x)-f(x)|\le\mathbb{P}_{n}H+PH$
almost surely. The associated optimal values are 
\begin{equation}
\psi^{*}=\inf_{x\in X}f(x),\qquad\hat{\psi}_{n}=\inf_{x\in X}\hat{f}_{n}(x).\label{eq:value-defs}
\end{equation}
For $\varepsilon,\delta\ge0$ define the expected and empirical $\varepsilon$-minimizer
sets 
\begin{equation}
S^{\varepsilon}=\{x\in X:f(x)\le\psi^{*}+\varepsilon\},\qquad\hat{S}_{n}^{\delta}=\{x\in X:\hat{f}_{n}(x)\le\hat{\psi}_{n}+\delta\}.\label{eq:near-minimizers}
\end{equation}
The set $\hat{S}_{n}^{\delta}$ is always nonempty when $\delta>0$,
and it is the natural empirical solution concept when exact minima
are not attained. The following convention records the empirical
minimizer membership assumption used in the transfer results.

\begin{assumption}[Empirical minimizer membership]\label{ass:emp-min-membership}
Let $(\delta_{n})$ be a nonnegative deterministic sequence and let
$\hat{x}_{n}:\Omega\to X$ be maps. We say that the empirical minimizer
membership assumption holds if there are events $A_{\mathrm{min},N}\in\mathcal{F}$,
$N\in\mathbb{N}$, with $A_{\mathrm{min},N}\subseteq A_{\mathrm{min},N+1}$ and
\[
\mathbb{P}\left(A_\mathrm{min}\right)=1,\quad A_\mathrm{min}=\bigcup_{N=1}^{\infty}A_{\mathrm{min},N},
\]
such that
\[
\hat{x}_{n}(\omega)\in\hat{S}_{n}^{\delta_{n}}(\omega)\qquad\text{for every }N\in\mathbb{N},\ \omega\in A_{\mathrm{min},N},\text{ and }n\ge N.
\]
\end{assumption}

\begin{remark}
Assumption \ref{ass:emp-min-membership} implies the following pathwise statement. On a full-probability event
$A_{\mathrm{min}}$, every sample path has a finite, path-dependent index
$N(\omega)$ after which the empirical minimizer membership condition $\hat{x}_{n}(\omega)\in\hat{S}_{n}^{\delta_{n}}(\omega)$ holds. Conversely, this pathwise formulation is equivalent to Assumption \ref{ass:emp-min-membership} when the corresponding finite-index sets can be chosen as measurable increasing events whose union has probability one.
\end{remark}

In the bounded-objective regime the uniform objective fluctuation
is 
\begin{equation}
\Delta_{n}=\lVert\hat{f}_{n}-f\rVert_{\infty}=\sup_{x\in X}\bigl|(\mathbb{P}_{n}-P)h_{x}\bigr|,\label{eq:Delta-def}
\end{equation}
where $h_{x}(\xi)=h(x,\xi)$ and $\mathbb{P}_{n}$ is the empirical
measure of $\xi_{1},\dots,\xi_{n}$. The organizing principle of the
paper is that once one controls $\Delta_{n}$, every value and $\varepsilon$-minimizer
rate follows by deterministic arguments.

For the empirical-process results we will use the class 
\[
\mathcal{H}=\{h_{x}=h(x,\cdot):x\in X\}.
\]
When the class is viewed as a subset of $\mathcal{G}$-measurable functions,
Assumption~\ref{ass:ep-meas}
will be stated explicitly whenever a standard empirical-process theorem
requires it. The abstract perturbation results of the next section
do not need this assumption.

\section{Main results}

\label{sec:main}

\subsection{Deterministic transfer from objective perturbations}

The first theorem is purely deterministic and is the mechanism by
which all probabilistic convergence statements are transferred from
$\Delta_{n}$ to optimal values and $\varepsilon$-minimizers.

\begin{theorem}[Bound transfers]\label{thm:deterministic} Let $f,\hat{f}_{n}\in\ell^{\infty}(X)$
be arbitrary. Define $\psi^{*}$, $\hat{\psi}_{n}$, $S^{\varepsilon}$,
and $\hat{S}_{n}^{\delta}$ as in \eqref{eq:value-defs}--\eqref{eq:near-minimizers},
and set
\[
\Delta_{n}=\lVert\hat{f}_{n}-f\rVert_{\infty}.
\]
Then for every $n$ and every $\varepsilon,\delta\ge0$ the following
statements hold. 
\begin{enumerate}[label=(\alph*),itemsep=0.25em]
\item \textbf{Value perturbation:} 
\begin{equation}
\lvert\hat{\psi}_{n}-\psi^{*}\rvert\le\Delta_{n}.\label{eq:value-perturbation}
\end{equation}
\item \textbf{$\varepsilon$-minimizer transfer:} 
\begin{equation}
\hat{S}_{n}^{\delta}\subseteq S^{2\Delta_{n}+\delta},\qquad S^{\varepsilon}\subseteq\hat{S}_{n}^{2\Delta_{n}+\varepsilon}.\label{eq:near-transfer}
\end{equation}
\item \textbf{Excess-risk bound:} if $\hat{x}_{n}\in\hat{S}_{n}^{\delta}$,
then 
\begin{equation}
f(\hat{x}_{n})-\psi^{*}\le2\Delta_{n}+\delta.\label{eq:excess-risk-bound}
\end{equation}
\item \textbf{Solution-set localization under sharp growth:} if $X\subseteq\mathbb{R}^{d}$,
$S^{0}\neq\varnothing$, and there exists $\alpha>0$ such that
\eqref{eq:sharp-growth} holds with exponent $\kappa$, then every
$\hat{x}_{n}\in\hat{S}_{n}^{\delta}$ satisfies
\begin{equation}
\operatorname{dist}(\hat{x}_{n},S^{0})^{\kappa}\le\alpha^{-1}(2\Delta_{n}+\delta).\label{eq:distance-bound}
\end{equation}
\end{enumerate}
\end{theorem}

\begin{remark} The theorem remains meaningful even when exact minimizers
do not exist. For example, on $X=[0,1]$ the bounded semialgebraic
function 
\[
f(x)=x+\mathbf{1}_{\{x=0\}}
\]
has infimum $0$ but no minimizer, since $f(x)>0$ for every $x\in[0,1]$.
In this case the value perturbation \eqref{eq:value-perturbation}
and the excess-risk bound \eqref{eq:excess-risk-bound} are still
valid and continue to control empirical $\varepsilon$-minimizers,
even though there is no exact minimizer set around which one could
localize. \end{remark}

\begin{corollary}[Stochastic transfers]\label{cor:generic-mode}
Let $(r_{n})$ be a positive deterministic sequence with $r_{n}\downarrow0$. 
\begin{enumerate}[label=(\roman*),itemsep=0.25em]
\item If $\Delta_{n}=O_{\mathbb{P}^{*}}(r_{n})$, then 
\[
\hat{\psi}_{n}-\psi^{*}=O_{\mathbb{P}^{*}}(r_{n}).
\]
Moreover, for every empirical $\delta_{n}$-minimizer sequence satisfying
Assumption~\ref{ass:emp-min-membership}, with a nonnegative deterministic
sequence $\delta_{n}=O(r_{n})$, 
\[
f(\hat{x}_{n})-\psi^{*}=O_{\mathbb{P}^{*}}(r_{n}).
\]
If, in addition, \eqref{eq:sharp-growth} holds, then 
\[
\operatorname{dist}(\hat{x}_{n},S^{0})=O_{\mathbb{P}^{*}}(r_{n}^{1/\kappa}).
\]
\item If $\Delta_{n}=O_{\mathrm{as}^{*}}(r_{n})$, then
\[
\lvert\hat{\psi}_{n}-\psi^{*}\rvert=O_{\mathrm{as}^{*}}(r_{n}).
\]
For every empirical $\delta_{n}$-minimizer sequence satisfying
Assumption~\ref{ass:emp-min-membership}, with a nonnegative deterministic
sequence $\delta_{n}=O(r_{n})$,
\[
f(\hat{x}_{n})-\psi^{*}=O_{\mathrm{as}^{*}}(r_{n})
\]
and, under \eqref{eq:sharp-growth},
\[
\operatorname{dist}(\hat{x}_{n},S^{0})=O_{\mathrm{as}^{*}}(r_{n}^{1/\kappa}).
\]
\item If $\mathbb{E}^{*}\Delta_{n}=O(r_{n})$, then
\[
\mathbb{E}^{*}\lvert\hat{\psi}_{n}-\psi^{*}\rvert=O(r_{n}).
\]
If, in addition, $\delta_{n}=O(r_{n})$ for a nonnegative deterministic
sequence $(\delta_{n})$ and, for an empirical $\delta_{n}$-minimizer sequence,
the events $A_{\mathrm{min},N}$ in Assumption~\ref{ass:emp-min-membership}
can be chosen so that
\[
\mathbb{P}(A_{\mathrm{min},n}^{c})=O(r_{n}),
\]
then
\[
\mathbb{E}^{*}\bigl[f(\hat{x}_{n})-\psi^{*}\bigr]=O(r_{n}),
\]
and, under \eqref{eq:sharp-growth}, 
\[
\mathbb{E}^{*}\operatorname{dist}(\hat{x}_{n},S^{0})^{\kappa}=O(r_{n}).
\]
\end{enumerate}
\end{corollary}

\subsection{Weak convergence}

The next assumption isolates the probabilistic input needed for the
weak-limit theorem.

\begin{assumption}[Weak convergence]\label{ass:weak} There exist
a deterministic sequence $\tau_{n}\uparrow\infty$ and a Borel
measurable $\ell^{\infty}(X)$-valued random element $G$ such that
\begin{equation}
\tau_{n}(\hat{f}_{n}-f)\rightsquigarrow G\qquad\text{in }\ell^{\infty}(X)\label{eq:weak-empirical}
\end{equation}
in the Hoffmann--Jørgensen sense. \end{assumption}

\begin{theorem}[Weak limit of the optimal value]\label{thm:weakvalue}
Assume that $f\in\ell^{\infty}(X)$, that $\hat{f}_{n}\in\ell^{\infty}(X)$
almost surely for all sufficiently large $n$, and that Assumption~\ref{ass:weak}
holds. Then 
\begin{equation}
\tau_{n}(\hat{\psi}_{n}-\psi^{*})\rightsquigarrow\iota_{f}'(G)=\lim_{\varepsilon\downarrow0}\inf_{x\in S^{\varepsilon}}G(x).\label{eq:weak-value}
\end{equation}
Moreover, 
\begin{equation}
\tau_{n}(\hat{\psi}_{n}-\psi^{*})-\iota_{f}'\bigl(\tau_{n}(\hat{f}_{n}-f)\bigr)=o_{\mathbb{P}^{*}}(1).\label{eq:weak-linearization}
\end{equation}
\end{theorem} 
\begin{corollary}[Donsker theorem implications]\label{cor:donsker}
Assume that $f\in\ell^{\infty}(X)$ and that $\hat{f}_{n}\in\ell^{\infty}(X)$
almost surely for all sufficiently large $n$. Suppose that, for a Borel measurable
random element $G$,
\[
\sqrt{n}(\hat{f}_{n}-f)\rightsquigarrow G\qquad\text{in }\ell^{\infty}(X).
\]
Then 
\begin{equation}
\sqrt{n}(\hat{\psi}_{n}-\psi^{*})\rightsquigarrow\lim_{\varepsilon\downarrow0}\inf_{x\in S^{\varepsilon}}G(x).\label{eq:sqrt-value-limit}
\end{equation}
In particular, 
\begin{equation}
\hat{\psi}_{n}-\psi^{*}=O_{\mathbb{P}^{*}}(n^{-1/2}),\label{eq:sqrt-rate}
\end{equation}
and for every empirical $\delta_{n}$-minimizer sequence satisfying
Assumption~\ref{ass:emp-min-membership}, with a nonnegative
deterministic sequence $\delta_{n}=O(n^{-1/2})$, 
\begin{equation}
f(\hat{x}_{n})-\psi^{*}=O_{\mathbb{P}^{*}}(n^{-1/2}).\label{eq:sqrt-excess}
\end{equation}
If, in addition \eqref{eq:sharp-growth} holds, then 
\begin{equation}
\operatorname{dist}(\hat{x}_{n},S^{0})=O_{\mathbb{P}^{*}}(n^{-1/(2\kappa)}).\label{eq:sqrt-distance}
\end{equation}
\end{corollary}

\subsection{Outer almost-sure and outer-mean rates from uniform bounds}

Outer almost-sure bounds on $\Delta_{n}$ transfer to values and
$\varepsilon$-minimizers. Outer-mean bounds transfer to values immediately
and to $\varepsilon$-minimizers under the corresponding membership-event
tail control.

\begin{theorem}[Outer almost-sure and outer-mean transfer]\label{thm:as-mean}
Let $b_{n}=\sqrt{\mathrm{LL}(n)/n}$. 
\begin{enumerate}[label=(\roman*),itemsep=0.25em]
\item If
\begin{equation}
\Delta_{n}=O_{\mathrm{as}^{*}}(b_{n}),\label{eq:blil-input}
\end{equation}
then
\begin{equation}
\lvert\hat{\psi}_{n}-\psi^{*}\rvert=O_{\mathrm{as}^{*}}(b_{n}).\label{eq:as-value}
\end{equation}
Moreover, for every empirical $\delta_{n}$-minimizer sequence satisfying
Assumption~\ref{ass:emp-min-membership}, with a nonnegative deterministic
sequence $\delta_{n}=O(b_{n})$,
\begin{equation}
f(\hat{x}_{n})-\psi^{*}=O_{\mathrm{as}^{*}}(b_{n}).\label{eq:as-excess}
\end{equation}
Under the same assumption, if \eqref{eq:sharp-growth} also holds, then
\begin{equation}
\operatorname{dist}(\hat{x}_{n},S^{0})=O_{\mathrm{as}^{*}}\bigl(b_{n}^{1/\kappa}\bigr).\label{eq:as-distance}
\end{equation}
\item If, for some deterministic sequence $\rho_{n}\downarrow0$, 
\begin{equation}
\mathbb{E}^{*}\Delta_{n}=O(\rho_{n}),\label{eq:mean-input}
\end{equation}
then 
\begin{equation}
\mathbb{E}^{*}\lvert\hat{\psi}_{n}-\psi^{*}\rvert=O(\rho_{n}).\label{eq:mean-value}
\end{equation}
If, in addition, $\delta_{n}=O(\rho_{n})$ for a nonnegative deterministic
sequence $(\delta_{n})$ and, for an empirical $\delta_{n}$-minimizer sequence,
the events $A_{\mathrm{min},N}$ in Assumption~\ref{ass:emp-min-membership}
can be chosen so that
\[
\mathbb{P}(A_{\mathrm{min},n}^{c})=O(\rho_{n}),
\]
then
\begin{equation}
\mathbb{E}^{*}\bigl[f(\hat{x}_{n})-\psi^{*}\bigr]=O(\rho_{n}).\label{eq:mean-excess}
\end{equation}
If \eqref{eq:sharp-growth} also holds, then 
\begin{equation}
\mathbb{E}^{*}\operatorname{dist}(\hat{x}_{n},S^{0})^{\kappa}=O(\rho_{n}).\label{eq:mean-distance}
\end{equation}
\end{enumerate}
\end{theorem}

\begin{remark}
The tail hypothesis
\[
\mathbb{P}(A_{\mathrm{min},n}^{c})=O(\rho_{n})
\]
in the outer-mean minimizer conclusions is automatic if the empirical
membership condition $\hat{x}_{n}\in\hat{S}_{n}^{\delta_{n}}$ holds
for all sufficiently large $n$, uniformly in $\omega$; then one may take $A_{\mathrm{min},N}=\Omega$
for all large $N$, and the complement probabilities are eventually zero.
This is the relevant case for deterministic or certified algorithms: if an
algorithm returns $\hat{x}_{n}$ together with a measurable lower bound
$L_{n}\le \hat{\psi}_{n}$ such that
\[
\hat{f}_{n}(\hat{x}_{n})-L_{n}\le \delta_{n},
\]
then automatically $\hat{x}_{n}\in\hat{S}_{n}^{\delta_{n}}$. This covers branch-and-bound algorithms and their variants; see, e.g., \citet[Chapter~5]{LocatelliSchoen2013} and \citet[Chapters~11--12]{Floudas2000}.
More generally, suppose there are events $B_{n}\in\mathcal{F}$ such that
$\hat{x}_{n}\in\hat{S}_{n}^{\delta_{n}}$ on $B_{n}$ and
\[
\sum_{k\ge n}\mathbb{P}(B_{k}^{c})=O(\rho_{n}).
\]
Here $B_{n}$ can be interpreted as the measurable success event on which
an implementation returns a $\delta_{n}$-empirical minimizer at sample
size $n$. The displayed tail bound says that the probability of at least
one future membership failure from time $n$ onward has the same order as
the desired outer-mean rate. In the root-$n$ outer-mean regime
$\rho_{n}=n^{-1/2}$, it is enough that
$\mathbb{P}(B_{n}^{c})\le Cn^{-3/2}$ for a constant $C$, since
$\sum_{k\ge n}k^{-3/2}=O(n^{-1/2})$. This applies, for instance, to a
randomized or time-limited implementation of such a certified scheme whose
certificate $\hat{f}_{n}(\hat{x}_{n})-L_{n}\le \delta_{n}$ fails with
probability at most $Cn^{-3/2}$.
\end{remark}

A convenient sufficient condition for the outer-mean input \eqref{eq:mean-input}
is provided by VC-subgraph complexity.

\begin{proposition}\label{prop:vc-mean}
Let $\mathcal{H}$ be a VC-subgraph class of measurable functions with
measurable envelope $H\in L_{2}(P)$, and assume that $\mathcal{H}$ is a
$P$-measurable class.
Then 
\begin{equation}
\mathbb{E}^{*}\sup_{h\in\mathcal{H}}\lvert(\mathbb{P}_{n}-P)h\rvert=O(n^{-1/2}).\label{eq:vc-mean}
\end{equation}
\end{proposition}

\subsection{Rates for definable classes with square-integrable envelopes}

The preceding theorems become directly applicable once the empirical-process
input is available. The next result packages the abstract rate theory
together with o-minimal definability and the VC-subgraph Donsker,
maximal-inequality, and LIL consequences available under a square-integrable
envelope.

\begin{theorem}[SAA rates for definable VC-subgraph classes]\label{thm:tame-main}
Retain the setup of Section~\ref{sec:setup}; in particular, assume
that $h_{x}$ is $\mathcal{G}$-measurable for each $x\in X$. Assume that
$X\subseteq\mathbb{R}^{d}$ and $\Xi\subseteq\mathbb{R}^{m}$ are
definable in an o-minimal structure, that the total subgraph of $h$
is definable in the same structure, that the class 
\[
\mathcal{H}=\{h_{x}=h(x,\cdot):x\in X\}
\]
satisfies Assumption~\ref{ass:ep-meas},
and that $\mathcal{H}$ admits a measurable envelope $H\in L_{2}(P)$.
Then $f\in\ell^{\infty}(X)$, $\Delta_{n}<\infty$ almost surely for
every $n$, and $\mathcal{H}$ is a VC-subgraph class. Moreover, there
exists a tight centered Gaussian process $G$, viewed
as a random element of $\ell^{\infty}(X)$, such that 
\[
\sqrt{n}(\hat{f}_{n}-f)\rightsquigarrow G\qquad\text{in }\ell^{\infty}(X),
\]
\begin{equation}
\mathbb{E}^{*}\Delta_{n}=O(n^{-1/2}),\label{eq:tame-mean}
\end{equation}
and 
\begin{equation}
\Delta_{n}=O\!\left(\sqrt{\frac{\log\log n}{n}}\right)\qquad\text{almost surely.}\label{eq:tame-blil}
\end{equation}
Consequently, Corollary~\ref{cor:donsker} applies, Theorem~\ref{thm:as-mean}(i)
applies with $b_{n}=\sqrt{\mathrm{LL}(n)/n}$, and the value conclusion of
Theorem~\ref{thm:as-mean}(ii) applies with $\rho_{n}=n^{-1/2}$; the
empirical-minimizer outer-mean conclusions in Theorem~\ref{thm:as-mean}(ii)
apply to sequences whose membership events satisfy
$\mathbb{P}(A_{\mathrm{min},n}^{c})=O(n^{-1/2})$. \end{theorem}

\section{Example applications}

\label{sec:applications}

This section records examples of discontinuous or
non-Lipschitz SAA optimization problems that fall outside the continuity-based
Banach-space framework of \citet{Banholzer2022}. In each example
the data take values in a fixed definable subset $\Xi\subseteq\mathbb{R}^{m}$,
the parameter set is a compact definable subset of Euclidean space,
and the loss is bounded and definable on parameter space times $\Xi$.
We assume throughout this section that $\mathcal{G}$ is the
Borel $\sigma$-algebra of $\Xi$, so the definable sections used below are
measurable. Hence each class satisfies the definability, boundedness,
and VC-subgraph parts of Theorem~\ref{thm:tame-main}. The remaining
task is to verify Assumption~\ref{ass:ep-meas}.

\subsection{Direct 0-1 classification for fixed-architecture neural networks}

Direct minimization of empirical $0$--$1$ loss for perceptrons
is studied by \citet{LiLin2007} and, using several exact and approximate
optimization algorithms, by \citet{NguyenSanner2013}. Section~7.2
of \citet{Westerhout2024} treats the same problem for a fixed
feedforward neural network. These are instances of a single template.
The simplest case is the perceptron loss 
\begin{equation}
h_{w,b}(z,y)=\mathbf{1}_{\{y(w^{\top}z+b)\le0\}},\label{eq:perceptron-loss}
\end{equation}
with data $\xi=(z,y)\in\Xi\subseteq\mathbb{R}^{m}\times\{-1,1\}$
and parameter $(w,b)\in\mathbb{R}^{m}\times\mathbb{R}$. More generally,
let $N_{\theta}:\mathbb{R}^{m}\to\mathbb{R}$ denote the score map
of a fixed finite network with parameter $\theta\in\Theta$. A representative example is the
ReLU architecture 
\begin{equation}
N_{\theta}(z)=a_{0}+\sum_{j=1}^{M}a_{j}\,\sigma(w_{j}^{\top}z+c_{j}),\qquad\sigma(r)=\max\{r,0\},\label{eq:relu-net}
\end{equation}
with parameter $\theta=(a_{0},a_{1},\dots,a_{M},w_{1},\dots,w_{M},c_{1},\dots,c_{M})$.
The associated classification loss is 
\begin{equation}
h_{\theta}(z,y)=\mathbf{1}_{\{y\,N_{\theta}(z)\le0\}}.\label{eq:nn-class-loss}
\end{equation}
For every fixed architecture whose activation functions are definable
in a common o-minimal structure, the score map $(\theta,z)\mapsto N_{\theta}(z)$
is definable on $\Theta\times\mathbb{R}^{m}$. The thresholded loss
\eqref{eq:nn-class-loss} is therefore bounded but generally discontinuous
in $\theta$.

\begin{proposition}\label{prop:classification} Let $\Theta$ be
a compact definable parameter set for a fixed finite architecture,
let $\Xi\subseteq\mathbb{R}^{m}\times\{-1,1\}$ be bounded and definable
in the same o-minimal structure, and assume that the score map $(\theta,z)\mapsto N_{\theta}(z)$
is definable. If, in addition, the class $\{h_{\theta}:\theta\in\Theta\}$
defined by \eqref{eq:nn-class-loss} satisfies Assumption~\ref{ass:ep-meas}, then it
is bounded and definable, its total subgraph is definable, Theorem~\ref{thm:def-vc}
yields that it is a VC-subgraph class, and Theorem~\ref{thm:tame-main}
applies. In particular, this covers the perceptron loss
\eqref{eq:perceptron-loss} and fixed ReLU networks of the form \eqref{eq:relu-net}
under the same measurability hypothesis. \end{proposition}

\begin{remark}
A sufficient condition for the measurability hypothesis in
Proposition~\ref{prop:classification} is the following. Since $\Theta$
is compact and metrizable, fix a countable dense subset $\Theta_{0}\subseteq\Theta$.
Assume that for every fixed $z\in\mathbb{R}^{m}$ the map $\theta\mapsto N_{\theta}(z)$
is continuous and that the tie set
\[
A_{0}=\{(z,y)\in\Xi: \text{there exists }\theta\in\Theta\text{ with }N_{\theta}(z)=0\}
\]
has $P(A_{0})=0$, where $A_{0}\in\mathcal{G}$ is automatic here
because $A_{0}$ is definable and $\mathcal{G}$ is the Borel $\sigma$-algebra
of $\Xi$. This is a global no-tie condition. Let $A=\Xi\setminus A_{0}$. Then for every $(z,y)\in A$,
every $\theta\in\Theta$, and every sequence $(\theta_{m})$ in $\Theta_{0}$
with $\theta_{m}\to\theta$, continuity and $N_{\theta}(z)\neq0$ imply
that $h_{\theta_{m}}(z,y)\to h_{\theta}(z,y)$. Define a modified class by
\[
\tilde{h}_{\theta}(z,y)=h_{\theta}(z,y)\mathbf{1}_{A}(z,y),\qquad(\theta,(z,y))\in\Theta\times\Xi.
\]
Then every $\tilde{h}_{\theta}$ vanishes on $A_{0}$, and for every
$(z,y)\in\Xi$, every $\theta\in\Theta$, and every sequence $(\theta_{m})$
in $\Theta_{0}$ with $\theta_{m}\to\theta$, we have $\tilde{h}_{\theta_{m}}(z,y)\to\tilde{h}_{\theta}(z,y)$:
on $A$ this is the preceding argument, while on $A_{0}$ all terms are
zero. Thus the modified class $\widetilde{\mathcal{H}}=\{\tilde{h}_{\theta}:\theta\in\Theta\}$
is pointwise measurable. By Remark~\ref{rem:ptwise-meas}, it therefore
satisfies Assumption~\ref{ass:ep-meas}. The original class $\mathcal{H}=\{h_{\theta}:\theta\in\Theta\}$
satisfies Assumption~\ref{ass:ep-meas} as well. Indeed, for every
$n\in\mathbb{N}$ the set $A^{n}\subseteq\Xi^{n}$ has $P^{n}(A^{n})=1$, and on
$A^{n}$ the supremum maps in Assumption~\ref{ass:ep-meas}(i) and (iii)
for $\mathcal{H}$ and $\widetilde{\mathcal{H}}$ coincide because $h_{\theta}=\tilde{h}_{\theta}$
on $A$ and hence $Ph_{\theta}=P\tilde{h}_{\theta}$ for every $\theta$. Likewise,
if $u=h_{\theta_{1}}-h_{\theta_{2}}$ and $\tilde{u}=\tilde{h}_{\theta_{1}}-\tilde{h}_{\theta_{2}}$,
then $u=\tilde{u}$ on $A$ and $\lVert u\rVert_{P,2}=\lVert\tilde{u}\rVert_{P,2}$,
so the same coincidence holds for the classes $\mathcal{H}_{\delta}$ and
$\mathcal{H}_{\infty}^{2}$. Hence measurability for the modified class
transfers to the original one in the relevant $P^{n}$-completions.
\end{remark}

\subsection{Threshold regression and discontinuous M-estimation}

Threshold regression and discontinuous change-point models provide
another natural source of discontinuous SAA objectives; see \citet{Hansen2000}
and \citet{KoulQianSurgailis2003}. A standard multivariate two-regime
regression model takes the form 
\begin{equation}
m_{\theta}(x,t)=x^{\top}\beta_{1}\,\mathbf{1}_{\{t\le s\}}+x^{\top}\beta_{2}\,\mathbf{1}_{\{t>s\}},\qquad\theta=(\beta_{1},\beta_{2},s)\in\mathbb{R}^{2p+1},\label{eq:threshold-reg}
\end{equation}
where $x\in\mathbb{R}^{p}$ is the regressor vector (including an
intercept if desired), $t\in\mathbb{R}$ is the threshold variable,
and $y\in\mathbb{R}$ is the response. This contains the usual sample-splitting
models as special cases; restrictions in which only part of the coefficient
vector changes across regimes are obtained by constraining coordinates
of $(\beta_{1},\beta_{2})$. For data $\xi=(x,t,y)$ and a definable
contrast $\rho:\mathbb{R}\to\mathbb{R}_{+}$, such as $\rho(r)=r^{2}$
or $\rho(r)=|r|$, the loss is 
\begin{equation}
h_{\theta}(x,t,y)=\rho\bigl(y-m_{\theta}(x,t)\bigr).\label{eq:threshold-loss}
\end{equation}
Even when $\rho$ is smooth, the map $\theta\mapsto h_{\theta}(x,t,y)$
is generally discontinuous in the threshold coordinate $s$ whenever
$t=s$.

\begin{proposition}\label{prop:threshold} Let $\Theta\subseteq\mathbb{R}^{2p+1}$
be compact and semialgebraic, let $\Xi\subseteq\mathbb{R}^{p}\times\mathbb{R}\times\mathbb{R}$
be bounded and semialgebraic, and let $\rho:\mathbb{R}\to\mathbb{R}_{+}$
be definable and bounded on the residual range
$\{y-m_{\theta}(x,t):(\theta,x,t,y)\in\Theta\times\Xi\}$. If, in addition, the
class $\{h_{\theta}:\theta\in\Theta\}$ defined by \eqref{eq:threshold-loss}
satisfies Assumption~\ref{ass:ep-meas},
then it is bounded and definable, its total subgraph is definable,
Theorem~\ref{thm:def-vc} yields that it is a VC-subgraph class, and
Theorem~\ref{thm:tame-main} applies to least-squares, least-absolute-deviation,
and other definable contrast criteria in threshold regression. \end{proposition}

\begin{remark}
For Proposition~\ref{prop:threshold}, the measurability hypotheses
are automatic under a simple one-sided approximation condition. Assume
that $\rho$ is continuous on the relevant bounded range and that there
exists a countable subclass $\Theta_{0}\subseteq\Theta$ with the property
that every $\theta=(\beta_{1},\beta_{2},s)\in\Theta$ admits a sequence
\[
\theta_{m}=(\beta_{1,m},\beta_{2,m},s_{m})\in\Theta_{0}
\]
with $\beta_{1,m}\to\beta_{1}$, $\beta_{2,m}\to\beta_{2}$, and $s_{m}\downarrow s$.
Then $h_{\theta_{m}}(x,t,y)\to h_{\theta}(x,t,y)$ for every $(x,t,y)\in\Xi$:
if $t<s$, eventually $t\le s_{m}$; if $t>s$, eventually $t>s_{m}$;
and if $t=s$, the convention $\mathbf{1}_{\{t\le s\}}$ is preserved
by the one-sided approximation $s_{m}\downarrow s$. Hence the class
is pointwise measurable, and therefore it satisfies Assumption~\ref{ass:ep-meas}. This covers
the standard box-constrained parameter sets $\Theta=B_{1}\times B_{2}\times[\underline{s},\bar{s}]$
when $B_{1}$ and $B_{2}$ admit countable dense subsets and the thresholds
are approximated from above by a countable set such as $(\mathbb{Q}\cap[\underline{s},\bar{s}))\cup\{\bar{s}\}$.
\end{remark}

\subsection{\texorpdfstring{Non-Lipschitz $\ell_{p}$-type objectives for rational
$0<p<1$}{Non-Lipschitz lp-type objectives for rational 0<p<1}}

A different obstruction to the continuity-based framework of \citet{Banholzer2022}
is the failure of Lipschitz continuity. \citet{YeYingShaoLiChen2017}
study robust and sparse $L_{p}$-norm support vector regression with
$0<p<1$, while \citet{BucciniEtAl2020} analyze related large-scale
regression problems with nonconvex loss and penalty. For the semialgebraic
verification below we fix rational $p\in(0,1)$. After the standard
elimination of slack variables, the support-vector-regression objective
considered in \citet{YeYingShaoLiChen2017} takes the form 
\[
\lambda\bigl(\lVert w\rVert_{p}^{p}+|b|^{p}\bigr)+C\frac{1}{n}\sum_{i=1}^{n}\bigl|y_{i}-(w^{\top}z_{i}+b)\bigr|.
\]
This is exactly the penalized SAA objective associated with the loss
\begin{equation}
h_{w,b}(z,y)=C\lvert y-(w^{\top}z+b)\rvert+\lambda\bigl(\lVert w\rVert_{p}^{p}+|b|^{p}\bigr),\qquad0<p<1.\label{eq:lp-loss}
\end{equation}
For fixed rational $p\in(0,1)$, the map $u\mapsto|u|^{p}$ is definable
already in the semialgebraic structure, but it is not Lipschitz at
the origin. Hence, on parameter domains meeting the relevant coordinate
zero sets, \eqref{eq:lp-loss} lies outside the Lipschitz-based assumptions
imposed in the continuity framework of \citet{Banholzer2022}, while
remaining tame in the present setting.

\begin{proposition}\label{prop:lp} Fix rational $p\in(0,1)$ and finite
constants $C,\lambda\ge0$. Let $X\subseteq\mathbb{R}^{m}\times\mathbb{R}$
be compact and semialgebraic, and let $\Xi\subseteq\mathbb{R}^{m}\times\mathbb{R}$
be bounded and semialgebraic. Then the class generated by \eqref{eq:lp-loss} is
pointwise measurable, bounded, and definable in the semialgebraic
structure, hence satisfies Assumption~\ref{ass:ep-meas}; its total subgraph is definable,
Theorem~\ref{thm:def-vc} yields that it is a VC-subgraph class, and
Theorem~\ref{thm:tame-main} applies to the associated penalized SAA
problem.\end{proposition}

\begin{remark} The examples in this section show why ``without smoothness''
in the title should be interpreted literally. In some cases the objective
jumps under arbitrarily small perturbations of the parameter, as in
direct $0$--$1$ classification and threshold regression. In others
the objective remains continuous but fails the Lipschitz or differentiability
properties built into continuity-based Banach-space analyses, as in
the $\ell_{p}$-type objectives with rational $0<p<1$. The natural
ambient space is therefore $\ell^{\infty}(X)$ rather than a Banach
space of continuous sample paths. \end{remark}

\section{Concluding remarks}

\label{sec:conclusion}

The main point of the manuscript is that the optimization-theoretic
results for SAAs extend beyond the smoothness-based formulations under
the correct abstraction. That abstraction is via the bounded-function
space $\ell^{\infty}(X)$ together with the infimum functional and
Hoffmann--Jørgensen outer probability. The weak-limit input comes
from the directional differentiability theorem of \citet{Westerhout2024};
the outer almost-sure and mean-rate inputs come from LILs and maximal inequalities;
and VC bounds under o-minimal definability provide a natural way to
verify the definability and VC-subgraph hypotheses for discontinuous
and nonsmooth classes. Envelope and measurability assumptions are
checked separately, with pointwise measurability used as a convenient
sufficient condition where needed.
A natural next step is to extend
the analysis beyond tame topological settings to nonsmooth settings without o-minimal definability assumptions and with relaxed measurability conditions.
We leave this to future work.

\appendix

\section{Proofs and auxiliary arguments}

\label{app:proofs}

\subsection{Deterministic perturbation lemmas}
\begin{proof}[Proof of Theorem~\ref{thm:deterministic}]
For part (a), fix $x\in X$. Since $\hat{f}_{n}(x)\ge f(x)-\Delta_{n}\ge\psi^{*}-\Delta_{n}$,
taking the infimum over $x$ yields $\hat{\psi}_{n}\ge\psi^{*}-\Delta_{n}$.
Interchanging the roles of $f$ and $\hat{f}_{n}$ gives $\psi^{*}\ge\hat{\psi}_{n}-\Delta_{n}$.
Hence \eqref{eq:value-perturbation} holds.

For part (b), let $x\in\hat{S}_{n}^{\delta}$. Then 
\[
f(x)\le\hat{f}_{n}(x)+\Delta_{n}\le\hat{\psi}_{n}+\delta+\Delta_{n}\le\psi^{*}+2\Delta_{n}+\delta,
\]
where the last step uses \eqref{eq:value-perturbation}. Thus $x\in S^{2\Delta_{n}+\delta}$,
proving the first inclusion in \eqref{eq:near-transfer}. The second
inclusion follows by symmetry.

Part (c) is an immediate restatement of the first inclusion in part
(b): if $\hat{x}_{n}\in\hat{S}_{n}^{\delta}$, then 
\[
\hat{x}_{n}\in S^{2\Delta_{n}+\delta},
\]
which is exactly \eqref{eq:excess-risk-bound}.

For part (d), apply the sharp-growth inequality \eqref{eq:sharp-growth}
to $x=\hat{x}_{n}$ and combine it with part (c): 
\[
\alpha\,\operatorname{dist}(\hat{x}_{n},S^{0})^{\kappa}\le f(\hat{x}_{n})-\psi^{*}\le2\Delta_{n}+\delta.
\]
This is \eqref{eq:distance-bound}. 
\end{proof}
\begin{proof}[Proof of Corollary~\ref{cor:generic-mode}]
The value bounds follow immediately from
$|\hat{\psi}_{n}-\psi^{*}|\le\Delta_{n}$. For empirical minimizers, let
$A_{\mathrm{min},N}$ and $A_{\mathrm{min}}$ be as in
Assumption~\ref{ass:emp-min-membership}. On $A_{\mathrm{min},N}$ and for
$n\ge N$, Theorem~\ref{thm:deterministic} gives
\[
f(\hat{x}_{n})-\psi^{*}\le2\Delta_{n}+\delta_{n},
\]
and, under sharp growth,
\[
\operatorname{dist}(\hat{x}_{n},S^{0})^{\kappa}
\le\alpha^{-1}(2\Delta_{n}+\delta_{n}).
\]
For the outer-probability assertion, fix $\eta>0$ and choose $N$ so large that
$\mathbb{P}(A_{\mathrm{min},N}^{c})<\eta/2$. For all $n\ge N$, the preceding
pointwise inequalities on $A_{\mathrm{min},N}$, together with
$\Delta_{n}=O_{\mathbb{P}^{*}}(r_{n})$ and $\delta_{n}=O(r_{n})$, give the
claimed $O_{\mathbb{P}^{*}}$ bounds.

For the outer almost-sure assertion, choose versions of the minimal measurable
majorants $\Delta_{n}^{*}$ such that $\limsup_{n}\Delta_{n}^{*}/r_{n}<\infty$
on a full-probability event $A_{\Delta}$. Since
$\lvert\hat{\psi}_{n}-\psi^{*}\rvert\le\Delta_{n}$ pointwise whenever the
objectives are bounded, the minimal measurable majorant of
$\lvert\hat{\psi}_{n}-\psi^{*}\rvert$ is bounded above by $\Delta_{n}^{*}$
a.s., and the outer almost-sure value bound follows. For empirical minimizers,
write $B_{f}=\sup_{x\in X}f(x)-\psi^{*}<\infty$. For every $n$, the deterministic
inequality on $A_{\mathrm{min},n}$ and the trivial bound
$f(\hat{x}_{n})-\psi^{*}\le B_{f}$ on its complement give
\[
f(\hat{x}_{n})-\psi^{*}\le2\Delta_{n}+\delta_{n}+B_{f}\mathbf{1}_{A_{\mathrm{min},n}^{c}}.
\]
Hence the minimal measurable majorant of the left-hand side is bounded a.s. by
$2\Delta_{n}^{*}+\delta_{n}+B_{f}\mathbf{1}_{A_{\mathrm{min},n}^{c}}$. On
$A_{\Delta}\cap A_{\mathrm{min}}$, the first term has order $r_{n}$, the
second has order $r_{n}$, and the indicator term is eventually zero. This proves
the outer almost-sure excess-risk bound. Under sharp growth, the same argument
with
\[
\operatorname{dist}(\hat{x}_{n},S^{0})\le\left\{\alpha^{-1}\left(2\Delta_{n}^{*}+\delta_{n}+B_{f}\mathbf{1}_{A_{\mathrm{min},n}^{c}}\right)\right\}^{1/\kappa}
\]
as a measurable majorant gives the stated distance bound.

For the outer-mean assertion, the value bound is again immediate. For the
empirical-minimizer bounds, assume that
$\mathbb{P}(A_{\mathrm{min},n}^{c})=O(r_{n})$. Since $f\in\ell^{\infty}(X)$,
$B_{f}=\sup_{x\in X}f(x)-\psi^{*}<\infty$. For each $n$,
\[
f(\hat{x}_{n})-\psi^{*}
\le 2\Delta_{n}+\delta_{n}+B_{f}\mathbf{1}_{A_{\mathrm{min},n}^{c}}.
\]
Applying monotonicity, positive homogeneity, and subadditivity of outer
expectation gives
\[
\mathbb{E}^{*}\bigl[f(\hat{x}_{n})-\psi^{*}\bigr]
\le2\mathbb{E}^{*}\Delta_{n}+\delta_{n}+B_{f}\mathbb{P}(A_{\mathrm{min},n}^{c})=O(r_{n}).
\]
Under sharp growth, the same argument applied to the distance bound, with
$B_{f}/\alpha$ on $A_{\mathrm{min},n}^{c}$, yields the stated
$\kappa$-moment bound.
\end{proof}

\subsection{Weak-limit arguments}
\begin{proof}[Proof of Theorem~\ref{thm:weakvalue}]
Let $\iota:\ell^{\infty}(X)\to\mathbb{R}$ be the infimum map from
\eqref{eq:iota-def}. By Theorem~\ref{thm:infdir}, $\iota$ is Hadamard
directionally differentiable at $f$ with derivative 
\[
\iota_{f}'(g)=\lim_{\varepsilon\downarrow0}\inf_{x\in S^{\varepsilon}}g(x).
\]
Since $\iota(\hat{f}_{n})=\hat{\psi}_{n}$ and $\iota(f)=\psi^{*}$,
after modifying $\hat{f}_{n}$ on a null set, we may for all sufficiently
large $n$ regard $\hat{f}_{n}$ as maps into $\ell^{\infty}(X)$. Theorem~\ref{thm:directional-delta} applies with
$U=\ell^{\infty}(X)$, $V=\mathbb{R}$, $\phi=\iota$, and $Y_{n}=\hat{f}_{n}$.
It yields 
\[
\tau_{n}\bigl(\iota(\hat{f}_{n})-\iota(f)\bigr)\rightsquigarrow\iota_{f}'(G)
\]
and also the first-order directional expansion
\[
\tau_{n}\bigl(\iota(\hat{f}_{n})-\iota(f)\bigr)-\iota_{f}'\bigl(\tau_{n}(\hat{f}_{n}-f)\bigr)=o_{\mathbb{P}^{*}}(1).
\]
Substituting $\iota(\hat{f}_{n})=\hat{\psi}_{n}$ and $\iota(f)=\psi^{*}$
gives \eqref{eq:weak-value} and \eqref{eq:weak-linearization}. 
\end{proof}
\begin{proof}[Proof of Corollary~\ref{cor:donsker}]
Apply Theorem~\ref{thm:weakvalue} with $\tau_{n}=\sqrt{n}$ to
obtain \eqref{eq:sqrt-value-limit}. Weak convergence in $\ell^{\infty}(X)$
implies boundedness in outer probability of the norms of the centered
process, and therefore 
\[
\Delta_{n}=\lVert\hat{f}_{n}-f\rVert_{\infty}=O_{\mathbb{P}^{*}}(n^{-1/2}).
\]
Corollary~\ref{cor:generic-mode}(i) then yields \eqref{eq:sqrt-rate},
\eqref{eq:sqrt-excess}, and, under sharp growth, \eqref{eq:sqrt-distance}. 
\end{proof}

\subsection{LIL for VC-subgraph classes}

\begin{theorem}[Yukich LIL]\label{thm:yukich} Let $\mathcal{H}$ be
a class of $\mathcal{G}$-measurable functions with envelope $H$, meaning that $\lvert h\rvert\le H$ for every $h\in\mathcal{H}$,
and suppose that $H\in L_{2}(P)$. Then, if 
\begin{equation}
\int_{0}^{1}\left(\frac{\log N(\varepsilon,\mathcal{H})}{\mathrm{LL}(\log N(\varepsilon,\mathcal{H}))}\right)^{1/2}\,d\varepsilon<\infty,\label{eq:yukich-integral}
\end{equation}
where, following \citet[Definition~1.2]{Yukich1990}, $N(\varepsilon,S,\mathcal{H})$
for a nonempty finite set $S=\{\xi_{1},\dots,\xi_{m}\}\subseteq\Xi$
denotes the least integer $N$ such that there exist $h_{1},\dots,h_{N}\in\mathcal{H}$
with the property that for every $h\in\mathcal{H}$ there is some
$j\in\{1,\dots,N\}$ satisfying 
\[
\frac{1}{m}\sum_{i=1}^{m}\bigl(h(\xi_{i})-h_{j}(\xi_{i})\bigr)^{2}\le\varepsilon^{2}\frac{1}{m}\sum_{i=1}^{m}H(\xi_{i})^{2},
\]
and where $N(\varepsilon,\mathcal{H})=\sup_{S}N(\varepsilon,S,\mathcal{H})$,
then $\mathcal{H}$ is a log-log class for $P$ in the sense of
\citet{Yukich1990}. Consequently, whenever the random variables
$\sup_{h\in\mathcal{H}}\lvert(\mathbb{P}_{n}-P)h\rvert$ are measurable,
\begin{equation}
\limsup_{n\to\infty}\sup_{h\in\mathcal{H}}\left|\left(\frac{n}{2\log\log n}\right)^{1/2}(\mathbb{P}_{n}-P)h\right|<\infty\qquad\text{almost surely.}\label{eq:yukich-blil}
\end{equation}
\end{theorem} 
\begin{proof}
This is \citet[Theorem~5]{Yukich1990}. 
\end{proof}
\begin{proof}[Proof of Proposition~\ref{prop:vc-yukich}]
Let $v<\infty$ denote the VC-subgraph dimension of $\mathcal{H}$.
By the VC covering theorem \citep[Theorem~2.6.7]{vdVW2023}, there
exist constants $C,a>0$, depending only on $v$, such that for every
finitely supported probability measure $Q$ and every $0<\eta<1$,
\begin{equation}
N\bigl(\eta\lVert H\rVert_{Q,2},\mathcal{H},L_{2}(Q)\bigr)\le C\eta^{-a}.\label{eq:vc-l2-cover}
\end{equation}
Increasing $C$ if necessary, we may take the centers of these covers
to belong to $\mathcal{H}$: start from an $\eta/2$-cover, discard
empty balls, and choose one element of $\mathcal{H}$ from each remaining
ball.

Fix $0<\varepsilon<1$ and let $S=\{\xi_{1},\dots,\xi_{m}\}\subseteq\Xi$
be a nonempty finite set. Write 
\[
Q_{S}=\frac{1}{m}\sum_{j=1}^{m}\delta_{\xi_{j}}.
\]
If $\lVert H\rVert_{Q_{S},2}=0$, then $H(\xi_{j})=0$ for every $j$,
and the envelope property implies that $h(\xi_{j})=0$ for all $h\in\mathcal{H}$
and all $j$. Hence $N(\varepsilon,S,\mathcal{H})=1$. If
$\lVert H\rVert_{Q_{S},2}=\infty$, then any single element of a nonempty
$\mathcal{H}$ is admissible in Yukich's finite-sample entropy definition,
since the right-hand side is infinite and the functions in $\mathcal{H}$
are real-valued on the finite set $S$; the empty-class case is trivial.
It remains to consider the case $0<\lVert H\rVert_{Q_{S},2}<\infty$.
Then any $L_{2}(Q_{S})$-cover of $\mathcal{H}$ with centers in $\mathcal{H}$
and radius $\varepsilon\lVert H\rVert_{Q_{S},2}$ gives an admissible family
in Definition~1.2 of \citet{Yukich1990}, because 
\[
\lVert h-h'\rVert_{Q_{S},2}^{2}=\frac{1}{m}\sum_{j=1}^{m}\bigl(h(\xi_{j})-h'(\xi_{j})\bigr)^{2}
\]
and 
\[
\lVert H\rVert_{Q_{S},2}^{2}=\frac{1}{m}\sum_{j=1}^{m}H(\xi_{j})^{2}.
\]
Therefore 
\[
N(\varepsilon,S,\mathcal{H})\le N\bigl(\varepsilon\lVert H\rVert_{Q_{S},2},\mathcal{H},L_{2}(Q_{S})\bigr)\le C\varepsilon^{-a}.
\]
Taking the supremum over all finite $S$ yields 
\[
N(\varepsilon,\mathcal{H})\le C\varepsilon^{-a},\qquad0<\varepsilon<1.
\]
Consequently 
\[
\log N(\varepsilon,\mathcal{H})\le c_{0}+a\log(1/\varepsilon),\qquad0<\varepsilon<1,
\]
for a suitable constant $c_{0}$, and the entropy integral \eqref{eq:yukich-integral}
is finite because its integrand is dominated near $0$ by a constant
multiple of $\sqrt{\log(1/\varepsilon)}$, which is integrable on
$(0,1)$. The conclusion \eqref{eq:yukich-blil} now follows from
Theorem~\ref{thm:yukich}. Since Assumption~\ref{ass:ep-meas}(iii)
ensures that, for each $n$, the random variable
$\sup_{h\in\mathcal{H}}\lvert(\mathbb{P}_{n}-P)h\rvert$ is measurable,
this is precisely the almost-sure bound \eqref{eq:vc-blil-rate}. 
\end{proof}

\subsection{Convergence mode transfer and definable classes}
\begin{proof}[Proof of Theorem~\ref{thm:as-mean}]
Part~(i) is Corollary~\ref{cor:generic-mode}(ii) with $r_{n}=b_{n}$.
Part~(ii) is Corollary~\ref{cor:generic-mode}(iii) with $r_{n}=\rho_{n}$. 
\end{proof}
\begin{proof}[Proof of Proposition~\ref{prop:vc-mean}]
By the outer-expectation maximal inequality for $P$-measurable
classes \citep[Theorem~2.14.1]{vdVW2023},
\[
\mathbb{E}^{*}\sup_{h\in\mathcal{H}}\lvert\mathbb{G}_{n}h\rvert\lesssim J\bigl(1,\mathcal{H}\mid H,L_{2}\bigr)\,\lVert H\rVert_{P,2},
\]
where $\mathbb{G}_{n}=\sqrt{n}(\mathbb{P}_{n}-P)$, the notation $A\lesssim B$
means $A\le CB$ for an absolute constant $C$, and 
\[
J\bigl(1,\mathcal{H}\mid H,L_{2}\bigr)=\sup_{Q}\int_{0}^{1}\sqrt{1+\log N\bigl(\varepsilon\lVert H\rVert_{Q,2},\mathcal{H},L_{2}(Q)\bigr)}\,d\varepsilon
\]
is the uniform entropy integral, with the supremum over finitely supported
probability measures $Q$. Since $\mathcal{H}$ is VC-subgraph, the
VC covering theorem \citep[Theorem~2.6.7]{vdVW2023} implies a polynomial
bound on the normalized covering numbers, and hence $J(1,\mathcal{H}\mid H,L_{2})<\infty$.
Thus $\mathbb{E}^{*}\sup_{h\in\mathcal{H}}\lvert\mathbb{G}_{n}h\rvert=O(1)$,
and dividing by $\sqrt{n}$ yields \eqref{eq:vc-mean}. 
\end{proof}
\begin{proof}[Proof of Theorem~\ref{thm:tame-main}]
Because $|h_{x}|\le H$ for every $x$, we have $|f(x)|\le PH$ and
$|\hat{f}_{n}(x)|\le\mathbb{P}_{n}H$ almost surely, so indeed $f\in\ell^{\infty}(X)$
and $\Delta_{n}<\infty$ almost surely. By Theorem~\ref{thm:def-vc},
the definability of the total subgraph implies that $\mathcal{H}$
is a VC-subgraph class. The assumed measurability hypotheses allow
Theorem~\ref{thm:vc-donsker} to be applied, so the empirical process
indexed by $\mathcal{H}$ is $P$-Donsker. Applying the pullback operator
from Remark~\ref{rem:pullback} to that process gives weak convergence of 
\[
x\longmapsto\sqrt{n}(\mathbb{P}_{n}-P)h_{x}=\sqrt{n}(\hat{f}_{n}-f)(x)
\]
in $\ell^{\infty}(X)$. Corollary~\ref{cor:donsker} then supplies
the announced root-$n$ outer-probability consequences for values
and $\varepsilon$-minimizers. The same measurability hypotheses permit
an application of Proposition~\ref{prop:vc-mean}, which gives \eqref{eq:tame-mean}
and hence the outer-mean value conclusion through Theorem~\ref{thm:as-mean}(ii);
the empirical-minimizer outer-mean conclusions require the membership-event
tail condition stated there. The measurability hypotheses also ensure that
the norms defining $\Delta_{n}$ are measurable, so Proposition~\ref{prop:vc-yukich}
gives \eqref{eq:tame-blil} and Theorem~\ref{thm:as-mean}(i) yields the
outer almost-sure consequences.
\end{proof}

\subsection{Verification of the examples}
\begin{proof}[Proof of Proposition~\ref{prop:classification}]
By assumption, the score map $(\theta,z)\mapsto N_{\theta}(z)$ is
definable on $\Theta\times\mathbb{R}^{m}$. Hence the set 
\[
\{(\theta,z,y)\in\Theta\times\Xi:yN_{\theta}(z)\le0\}
\]
is definable, and therefore the classification loss \eqref{eq:nn-class-loss}
is definable. Boundedness is immediate because the loss is $\{0,1\}$-valued.
It follows that the total subgraph of the class is definable, so Theorem~\ref{thm:def-vc}
yields the VC-subgraph conclusion. The assumed measurability hypotheses
then allow Theorem~\ref{thm:tame-main} to supply the stated SAA
consequences. The perceptron model \eqref{eq:perceptron-loss} is
the special case $N_{\theta}(z)=w^{\top}z+b$, which is semialgebraic. 
\end{proof}
\begin{proof}[Proof of Proposition~\ref{prop:threshold}]
The regression map \eqref{eq:threshold-reg} is definable because
it is assembled from affine maps and the definable threshold predicates
$\{t\le s\}$ and $\{t>s\}$. Therefore $(\theta,x,t,y)\mapsto y-m_{\theta}(x,t)$
is definable on $\Theta\times\Xi$, and so is $h_{\theta}(x,t,y)=\rho(y-m_{\theta}(x,t))$
because $\rho$ is definable on the relevant bounded range. Boundedness
is assumed. Hence the total subgraph of the class is definable, and
the assumed measurability hypotheses allow Theorems~\ref{thm:def-vc}
and \ref{thm:tame-main} to apply. 
\end{proof}
\begin{proof}[Proof of Proposition~\ref{prop:lp}]
Write $p=r/s$ in lowest terms with integers $0<r<s$. The graph
of the map $u\mapsto|u|^{p}$ is definable in the semialgebraic structure
because it is the union of the two semialgebraic sets
\[
\{(u,t):u\ge0,\ t\ge0,\ t^{s}=u^{r}\}
\cup\{(u,t):u<0,\ t\ge0,\ t^{s}=(-u)^{r}\}.
\]
It follows that $(w,b)\mapsto\lVert w\rVert_{p}^{p}+|b|^{p}$ is definable
on the compact set $X$. The residual term $(w,b,z,y)\mapsto|y-(w^{\top}z+b)|$
is definable as well. Therefore the full loss \eqref{eq:lp-loss}
is definable on $X\times\Xi$ and bounded because both $X$ and $\Xi$
are bounded.

To verify pointwise measurability, choose a countable dense subset
$X_{0}\subseteq X$, which exists because $X$ is compact and metrizable.
For each fixed $(z,y)\in\Xi$, the map $(w,b)\mapsto h_{w,b}(z,y)$
is continuous on $X$: the residual term is continuous, and $u\mapsto|u|^{p}$
is continuous on $\mathbb{R}$ for every $p>0$. Hence if $x=(w,b)\in X$
and $x_{m}=(w_{m},b_{m})\in X_{0}$ satisfy $x_{m}\to x$, then $h_{x_{m}}(z,y)\to h_{x}(z,y)$
for every $(z,y)\in\Xi$. Thus the class is pointwise measurable.

Theorem~\ref{thm:def-vc} then yields the VC-subgraph property, and
Theorem~\ref{thm:tame-main} gives the conclusion. 
\end{proof}

\end{document}